\documentclass[12pt,twoside,reqno]{amsart}
\usepackage[colorlinks=false,citecolor=blue]{hyperref}
\usepackage{mathptmx, amsmath, amssymb, amsfonts, amsthm, enumerate, color, mathrsfs}
\usepackage{graphicx}
\usepackage{epstopdf}

\newcommand{\vertiii}[1]{{\left\vert\kern-0.25ex\left\vert\kern-0.25ex\left\vert #1 
    \right\vert\kern-0.25ex\right\vert\kern-0.25ex\right\vert}}
\newtheorem{theorem}{Theorem}[section]
\newtheorem{lemma}{Lemma}[section]
\newtheorem{remark}{Remark}[section]

\newtheorem{corollary}{Corollary}[section]

\numberwithin{equation}{section}

\begin{document}
\title{Integral Representations and Sharp Numerical Radius Inequalities for Hilbert Space Operators}
\author{Wasim Audeh, Fuad Kittaneh, and Hamid Reza Moradi}

\begin{abstract}
We study the numerical radius \(\omega(T)\) of a bounded linear operator \(T\) on a complex Hilbert space. Write \(T=A+iB\), where \(A\) and \(B\) are the real and imaginary parts of \(T\). A central result is the exact formula
\[
\omega(T)=\frac{\pi}{2}\sup_{\|x\|=1}\int_0^1
\left|\left\langle (A\cos \pi t+B\sin \pi t)x,x\right\rangle\right|\,dt .
\]
For every \(n\in\mathbb N\), we also prove the power form
\[
\omega^{2n}(T)=\frac{4^n}{\binom{2n}{n}}
\sup_{\|x\|=1}\int_0^1
\left\langle (A\cos \pi t+B\sin \pi t)x,x\right\rangle^{2n}\,dt .
\]
This yields sharp upper bounds; for \(n=1\), together with the matching integral lower bound, it recovers Kittaneh's classical two-sided inequality
\[
\frac14\|T^*T+TT^*\|\le \omega^2(T)\le \frac12\|T^*T+TT^*\|.
\]
Furthermore, we obtain exact numerical radius formulas for some block operator matrices, including a \(3\times3\) nilpotent matrix and longer alternating paths. The factor \(\pi/2\) is best possible.
\end{abstract}

\maketitle

\pagestyle{myheadings} \markboth{\centerline{W. Audeh, F. Kittaneh \& H. R. Moradi}}
{\centerline{}} \bigskip \bigskip 

\section{Introduction}
Let \(\mathbb{H}\) be a complex Hilbert space, that is, a vector space over
\(\mathbb{C}\) equipped with an inner product \(\langle\cdot,\cdot\rangle\)
that is complete with respect to the norm
\(\|x\|=\sqrt{\langle x,x\rangle}\). We mean by \(\mathbb{B}(\mathbb{H})\)
the \(C^*\)-algebra of all bounded linear operators on \(\mathbb{H}\), where an operator \(T:\mathbb{H}\to\mathbb{H}\) is bounded if there exists a constant \(C>0\) such that \(\|Tx\|\le C\|x\|\) for all \(x\in\mathbb{H}\).
The smallest such constant is the operator norm
\(\|T\|=\sup_{\|x\|=1}\|Tx\|\). We denote the zero operator and the identity operator by \(O\) and \(I\), respectively. For \(T\in\mathbb{B}(\mathbb{H})\),
the numerical radius of \(T\) is represented by
\[
\omega(T)=\sup\{|\langle Tx,x\rangle|:x\in\mathbb{H},\ \|x\|=1\}.
\]
This number calculates the largest possible absolute value of the quadratic form \(\langle Tx,x\rangle\) over unit vectors. It is well known that
\(\omega(\cdot)\) is a norm on \(\mathbb{B}(\mathbb{H})\) and satisfies the
essential equivalence
\begin{equation}\label{intro:basic}
\frac{1}{2}\|T\|\le\omega(T)\le\|T\|.
\end{equation}
These inequalities indicate that the numerical radius is a refinement of the
operator norm and have inspired extensive research aiming to sharpen and generalize them.

A main tool in analyzing the numerical radius is the Cartesian
decomposition. Every \(T\in\mathbb{B}(\mathbb{H})\) can be written uniquely as
\[
T=A+iB,
\]
where
\[
A=\Re T=\frac{T+T^*}{2},\qquad B=\Im T=\frac{T-T^*}{2i}.
\]
The operators \(A\) and \(B\) are self-adjoint, meaning \(A=A^*\) and
\(B=B^*\). For a self-adjoint operator \(S\), the numerical radius coincides
with the operator norm: \(\omega(S)=\|S\|\). From the definition of
\(\omega(T)\) one observes immediately that
\[
\|\Re T\|\le\omega(T),\qquad \|\Im T\|\le\omega(T).
\]

Over the years, several refinements of \eqref{intro:basic} have been
obtained. A famous result due to Kittaneh \cite{Kittaneh_Studia_2005}
asserts that
\begin{equation}\label{intro:kittaneh}
\frac{1}{4}\|T^*T+TT^*\|\le\omega^2(T)\le\frac{1}{2}\|T^*T+TT^*\|.
\end{equation}
Related commutator and anticommutator norm inequalities with applications to numerical radius bounds were recently obtained in \cite{SababhehAudehMoradi_MJM_2025}.

Another well known upper bound is provided by \cite{Kittaneh_Studia_2003}
\begin{equation*}
\omega(T)\le\frac{1}{2}\|\,|T|+|T^*|\,\|\le\frac{1}{2}\bigl(\|T\|+\|T^2\|^{1/2}\bigr),
\end{equation*}
where \(|T|=(T^*T)^{1/2}\) is the absolute value of \(T\). For every
\(n\in\mathbb{N}\), the numerical radius also satisfies the power
inequality \(\omega(T^n)\le\omega^n(T)\).

The numerical radius is also intrinsically linked to block operator
matrices, that is, matrices whose entries are bounded linear operators. For instance, the \(2\times2\) block operator matrix
\[
\begin{bmatrix}O&T\\T^*&O\end{bmatrix}
\]
acts on \(\mathbb{H}\oplus\mathbb{H}\) and satisfies the norm identity
\begin{equation*}
\left\|\begin{bmatrix}O&T\\T^*&O\end{bmatrix}\right\|=\|T\|.
\end{equation*}
Off-diagonal block operator matrices of the form
\(\begin{bmatrix}O&A\\B&O\end{bmatrix}\) have proven to be exceptionally
versatile in developing sharp numerical radius bounds. The inner product on
\(\mathbb{H}\oplus\mathbb{H}\) is
\[
\left\langle
\begin{pmatrix}
x_1\\
y_1
\end{pmatrix},
\begin{pmatrix}
x_2\\
y_2
\end{pmatrix}
\right\rangle
=
\langle x_1,x_2\rangle+\langle y_1,y_2\rangle.
\]

Recent work on \(2\times2\) operator matrices includes \cite{Conde_GMJ_2024, HirzallahKittanehShebrawi, HirzallahKittanehShebrawi2}. 
More recent investigations of upper and lower numerical radius bounds for certain operator matrices, including norm and numerical radius inequalities, appear in \cite{AudehMoradiSababheh_JMI_2026, AudehMoradiSababheh_Filomat_2025, Kittaneh_LAMA_2025I, Sababheh_MIA_2024}. For numerical radius inequalities for general \(n\times n\) operator matrices, 
see \cite{Abu-Omar_LAA_2015}.

More recently, considerable research has focused on integral forms of
numerical radius inequalities. These representations provide refined bounds by integrating operator families over continuous parameter ranges. For instance, it was established in \cite{intro:integral} that
\begin{equation}\label{intro:integral}
\omega(A+B)\le2\int_0^1\omega((1-t)A+tB)\,dt\le\omega(A)+\omega(B),
\end{equation}
which refines the triangle inequality for the numerical radius.
Also, integral bounds involving the absolute values $\left| T \right|$ and $\left| {{T}^{*}} \right|$ have been derived, such as \cite{Moradi_LAMA_2021,Sababheh_LAMA_2021}
\begin{equation*}
\omega^2(T)\le\left\|\int_0^1\bigl(t|T|+(1-t)|T^*|\bigr)^2dt\right\|
\le\frac{1}{2}\|\,|T|^2+|T^*|^2\,\|.
\end{equation*}
For related H\"older-type numerical radius inequalities and further refinements involving
\(|T|\) and \(|T^*|\), see \cite{AudehMoradiSababheh_LAA_2024}. For sharp numerical radius inequalities associated with the Cartesian decomposition, see also \cite{Kittaneh_Cartesian_2015}.

In this paper, we extend and refine numerical radius inequalities by synthesizing integral techniques, block operator matrix transformations, and operator norm estimates. We establish three main lines of results. First, we study a \(3\times 3\) nilpotent block operator matrix and a related family of \(2\times 2\) off-diagonal matrices, obtaining exact formulas for their numerical radii. We show that the square of the numerical radius of the \(3\times 3\) matrix is exactly twice the square of the numerical radius of the \(2\times 2\) matrix. Second, we extend this to longer alternating paths. For \(m\ge 2\), let \(\mathbb{T}_m\) denote the \(m\times m\) block operator matrix with zero diagonal and alternating first superdiagonal entries \(T,T^*,T,T^*,\ldots\), that is,
\[
(\mathbb{T}_m)_{j,j+1}=
\begin{cases}
T, & j \text{ odd},\\
T^*, & j \text{ even},
\end{cases}
\qquad 1\le j\le m-1,
\]
with all other blocks equal to \(O\); in particular, \(\mathbb{T}_3\) is the \(3\times 3\) matrix mentioned above. We compute the numerical radius for every path length \(m\), obtaining the formula
\[
\omega(\mathbb{T}_m)=\|T\|\cos\frac{\pi}{m+1}.
\]
 Third, we use the Cartesian decomposition to show exact integral representations for
\(\omega(T)\) and \(\omega^{2n}(T)\). The central analytic result is the exact formula
\[
\omega(T)=\frac{\pi}{2}\sup_{\|x\|=1}\int_0^1
\left|\left\langle(A\cos\pi t+B\sin\pi t)x,x\right\rangle\right|dt.
\]
We also establish the power form
\[
\omega^{2n}(T)=\frac{4^n}{\binom{2n}{n}}
\sup_{\|x\|=1}\int_0^1
\left\langle(A\cos\pi t+B\sin\pi t)x,x\right\rangle^{2n}dt.
\]
As a result, we obtain the sharp upper bound
\[
\omega^{2n}(T)\le\frac{4^n}{\binom{2n}{n}}
\left\|\int_0^1(A\cos\pi t+B\sin\pi t)^{2n}dt\right\|.
\]
For \(n=1\), this gives the upper half of Kittaneh's classical inequality; the matching lower bound follows from the integral sandwich, recovering the full two-sided inequality \eqref{intro:kittaneh}. The factor \(\pi/2\) in the first-order
representation is best possible, as shown by the identity operator \(T=I\).
The max-phase representation \(\omega(T)=\max_{\theta\in[0,2\pi)}\|\Re(e^{i\theta}T)\|\) is closely
connected to the results in \cite{Kittaneh_Moslehian_Yamazaki}.

Our proofs use integral averaging, the Cartesian decomposition, block
operator matrices, and the spectral theorem for self-adjoint operators. For a self-adjoint operator \(S\), the spectral theorem gives a spectral measure \(E\) such that \(S=\int s\,dE_s\). For a unit vector \(x\), the measure
\(\mu_x(\cdot)=\langle E(\cdot)x,x\rangle\) is a probability measure, and
Jensen's inequality for a convex function \(f\) reads
\[
f\!\left(\int s\,d\mu_x(s)\right)\le\int f(s)\,d\mu_x(s).
\]
The integrals occurring in this paper are Bochner integrals, that is, the
extension of the Lebesgue integral to functions taking values in a Banach
space; for a continuous operator-valued function \(t\mapsto R(t)\) on
\([0,1]\), the integral \(\int_0^1R(t)\,dt\) is a well-defined bounded linear
operator. In the study of block matrices, we also use vector states and
normal states. A vector state on \(\mathbb{B}(\mathbb{H})\) is a functional
of the form \(\psi_x(S)=\langle Sx,x\rangle\) for \(S\in\mathbb{B}(\mathbb{H})\),
where \(x\in\mathbb{H}\) is a unit vector. A normal state is a positive
linear functional \(\varphi\) on \(\mathbb{B}(\mathbb{H})\) with
\(\|\varphi\|=1\) that is \(\sigma\)-weakly continuous; every normal state
can be written as \(\varphi(S)=\sum_j\lambda_j\langle Sx_j,x_j\rangle\),
where \(\{x_j\}\) is an orthonormal sequence and \(\lambda_j\ge0\) with
\(\sum_j\lambda_j=1\). Finally, in Corollary~\ref{cor:alternating-path} we
use the \(m\times m\) tridiagonal matrix \(C_m\) with zero diagonal and
\(1/2\) on the sub- and superdiagonals; its largest eigenvalue is
\(\cos\frac{\pi}{m+1}\).

The results significantly extend the scope of existing numerical radius
theory and provide new tools for estimating \(\omega(T)\) in terms of
integral operators.

\section{Exact Numerical Radii and Integral Representations}

In this section, we give our main results. We start with an exact formula for a block matrix. Then, we study a special \(3\times3\) matrix and a family of \(2\times2\) matrices. Next, we establish an integral identity for
trigonometric polynomials. Finally, we utilize the Cartesian decomposition to give integral representations and bounds for the numerical radius.

Off-diagonal block matrices are important in numerical radius problems. The next lemma provides an exact formula for the numerical radius of a general off-diagonal \(2\times2\) block matrix. This formula changes the problem into a supremum over unit vectors. It also illustrates that the special self-adjoint case gives the operator norm.
\begin{lemma}\label{lem:offdiag-exact}
Let \(A,B\in\mathbb{B}(\mathbb{H})\). Then
\begin{equation}\label{eq:offdiag-exact}
\omega\!\left(\begin{bmatrix}O&A\\B&O\end{bmatrix}\right)
=
\frac{1}{2}\,
\sup_{\|u\|=\|v\|=1}
\bigl|\langle Av,u\rangle+\langle Bu,v\rangle\bigr| .
\end{equation}
In particular, if \(B=A^{*}\), then
\begin{equation}\label{eq:offdiag-selfadjoint}
\omega\!\left(\begin{bmatrix}O&A\\A^{*}&O\end{bmatrix}\right)=\|A\| .
\end{equation}
\end{lemma}
\begin{proof}
Let \(M=\begin{bmatrix}O&A\\B&O\end{bmatrix}\) and let $z=\begin{pmatrix}x\\y\end{pmatrix}\in\mathbb H\oplus\mathbb H$ with \(\|x\|^{2}+\|y\|^{2}=1\). Then
\[
Mz=
\begin{pmatrix}
Ay\\
Bx
\end{pmatrix},
\qquad
\langle Mz,z\rangle
=
\langle Ay,x\rangle+\langle Bx,y\rangle .
\]
Writing \(x=\cos\alpha\,u\), \(y=\sin\alpha\,v\) with
\(\|u\|=\|v\|=1\) and \(\alpha\in[0,\pi/2]\), we obtain
\[
\langle Mz,z\rangle
=
\sin\alpha\cos\alpha\bigl(\langle Av,u\rangle+\langle Bu,v\rangle\bigr).
\]
Since \(\sin\alpha\cos\alpha\le\tfrac12\) with equality at
\(\alpha=\pi/4\), taking the supremum over \(\alpha,u,v\) yields
\eqref{eq:offdiag-exact}.

For the second statement, take \(B=A^{*}\). Then
\[
\langle Bu,v\rangle=\langle A^{*}u,v\rangle=\langle u,Av\rangle
=\overline{\langle Av,u\rangle},
\]
so
\[
|\langle Av,u\rangle+\langle Bu,v\rangle|
=
2\,|\mathrm{Re}\,\langle Av,u\rangle| .
\]
The supremum of \(|\mathrm{Re}\,\langle Av,u\rangle|\) over unit \(u,v\)
equals \(\|A\|\), because for any unit \(v\) one may choose
\(u=Av/\|Av\|\) when \(Av\neq0\), and the reverse inequality
\(|\mathrm{Re}\,\langle Av,u\rangle|\le\|A\|\) is immediate. Substituting
into \eqref{eq:offdiag-exact} shows \eqref{eq:offdiag-selfadjoint}.
\end{proof}

The following theorem studies a \(3\times3\) nilpotent block matrix and a related family of \(2\times2\) off-diagonal matrices. It shows exact formulas for their numerical radii. It also shows that the square of the \(3\times3\) numerical radius equals twice the square of the \(2\times2\) numerical radius. This provides an exact identity, not just an inequality.
\begin{theorem}\label{thm:deep_3x3_operator_matrix}
Let \(T\in\mathbb{B}(\mathbb{H})\) with \(T\neq O\). Define the \(3\times3\)
block operator matrix
\[
\mathbb T=\begin{bmatrix}O&T&O\\O&O&T^{*}\\O&O&O\end{bmatrix}
\in\mathbb{B}(\mathbb{H}\oplus\mathbb{H}\oplus\mathbb{H}),
\]
and, for \(t\in[0,1]\), the \(2\times2\) off-diagonal operator matrix
\[
M_{t}=\begin{bmatrix}O&tT\\(1-t)T^{*}&O\end{bmatrix}
\in\mathbb{B}(\mathbb{H}\oplus\mathbb{H}).
\]
Then the following assertions hold.
\begin{enumerate}
\item[(i)] \(\displaystyle
\omega(\mathbb T)=\frac{1}{\sqrt2}\,\|T\|\).
\item[(ii)] For every \(t\in[0,1]\),
\[
\omega(M_{t})
=
\frac12\,
\sup_{\|u\|=\|v\|=1}
\bigl|t\,\langle Tv,u\rangle+(1-t)\,\overline{\langle Tv,u\rangle}\bigr|
=\frac{\|T\|}{2}.
\]
In particular,
\(\displaystyle \sup_{t\in[0,1]}\omega(M_{t})=\frac{\|T\|}{2}\).
\item[(iii)] For every \(t\in[0,1]\),
\[
\omega^{2}(\mathbb T)=2\,\omega^{2}(M_{t}).
\]
\item[(iv)] For every continuous weight
\(f\colon[0,1]\to[0,\infty)\) with \(\int_{0}^{1}f(t)\,dt=1\),
\[
\omega^{2}(\mathbb T)=2\int_{0}^{1}f(t)\,\omega^{2}(M_{t})\,dt.
\]
\end{enumerate}
\end{theorem}
\begin{proof}
Throughout, \(\mathbb H\oplus\mathbb H\oplus\mathbb H\) is equipped with the usual inner product $
\left\langle
\begin{pmatrix}
x_1\\
y_1\\
z_1
\end{pmatrix},
\begin{pmatrix}
x_2\\
y_2\\
z_2
\end{pmatrix}
\right\rangle
=
\langle x_1,x_2\rangle+\langle y_1,y_2\rangle+\langle z_1,z_2\rangle$.

\medskip
\textbf{(i)}
Let $
w=
\begin{pmatrix}
x\\
y\\
z
\end{pmatrix}
\in\mathbb H\oplus\mathbb H\oplus\mathbb H $ with
\(\|x\|^{2}+\|y\|^{2}+\|z\|^{2}=1\). A direct calculation gives
\[
\mathbb T
\begin{pmatrix}
x\\
y\\
z
\end{pmatrix}
=
\begin{pmatrix}
Ty\\
T^*z\\
O
\end{pmatrix},
\]
and therefore
\[
\left\langle
\mathbb T
\begin{pmatrix}
x\\
y\\
z
\end{pmatrix},
\begin{pmatrix}
x\\
y\\
z
\end{pmatrix}
\right\rangle
=
\langle Ty,x\rangle+\langle T^*z,y\rangle
=
\langle Ty,x\rangle+\langle z,Ty\rangle ,
\]
where the second equality uses the definition of the adjoint. By the triangle inequality and the Cauchy-Schwarz inequality, we have
\begin{equation}\label{eq:cs-3x3}
|\langle\mathbb T w,w\rangle|
\le
\|Ty\|\,\|x\|+\|z\|\,\|Ty\|
=
\|Ty\|\bigl(\|x\|+\|z\|\bigr)
\le
\|T\|\,\|y\|\,\bigl(\|x\|+\|z\|\bigr).
\end{equation}
Since
\[
\bigl(\|x\|+\|z\|\bigr)^{2}
\le
2\bigl(\|x\|^{2}+\|z\|^{2}\bigr)
=
2\bigl(1-\|y\|^{2}\bigr),
\]
we get
\[
\|y\|\bigl(\|x\|+\|z\|\bigr)
\le
\|y\|\sqrt{2(1-\|y\|^{2})}
=
\sqrt{2}\,\sqrt{\|y\|^{2}\bigl(1-\|y\|^{2}\bigr)}
\le
\sqrt{2}\cdot\frac12
=
\frac{1}{\sqrt2},
\]
where the last inequality uses \(s(1-s)\le\tfrac14\) for
\(s=\|y\|^{2}\in[0,1]\). Substituting into \eqref{eq:cs-3x3} and taking the
supremum over all unit \(w\),
\begin{equation}\label{eq:upper-3x3}
\omega(\mathbb T)\le\frac{\|T\|}{\sqrt2}.
\end{equation}

For the reverse inequality, let \(\varepsilon>0\) and choose a unit vector
\(e\in\mathbb{H}\) with \(\|Te\|>\|T\|-\varepsilon\). Set
\[
y=\frac{e}{\sqrt2},\qquad
x=z=\frac{Te}{2\,\|Te\|},
\]
so that
\[
\|x\|^{2}+\|y\|^{2}+\|z\|^{2}
=
\frac14+\frac12+\frac14
=
1.
\]
Then

{\small
\[
\left\langle
\mathbb T
\begin{pmatrix}
x\\
y\\
z
\end{pmatrix},
\begin{pmatrix}
x\\
y\\
z
\end{pmatrix}
\right\rangle=
\Bigl\langle T\frac{e}{\sqrt2},\,\frac{Te}{2\|Te\|}\Bigr\rangle
+
\Bigl\langle \frac{Te}{2\|Te\|},\,T\frac{e}{\sqrt2}\Bigr\rangle
=
\frac{\|Te\|^{2}}{2\sqrt2\,\|Te\|}
+
\frac{\|Te\|^{2}}{2\sqrt2\,\|Te\|}
=
\frac{\|Te\|}{\sqrt2}
>
\frac{\|T\|-\varepsilon}{\sqrt2}.
\]}
Therefore, \(\omega(\mathbb T)\ge(\|T\|-\varepsilon)/\sqrt2\). Letting
\(\varepsilon\to0^{+}\) yields \(\omega(\mathbb T)\ge\|T\|/\sqrt2\),
which together with \eqref{eq:upper-3x3} gives
\(\omega(\mathbb T)=\|T\|/\sqrt2\). This confirms (i).

\medskip
\textbf{(ii)}
Fix \(t\in[0,1]\). Applying Lemma \ref{lem:offdiag-exact} with \(A=tT\) and
\(B=(1-t)T^{*}\), we obtain
\begin{equation*}
\omega(M_{t})
=
\frac12\,
\sup_{\|u\|=\|v\|=1}
\bigl|t\,\langle Tv,u\rangle+(1-t)\,\langle T^{*}u,v\rangle\bigr| .
\end{equation*}
Since
\[
\langle T^{*}u,v\rangle=\langle u,Tv\rangle=\overline{\langle Tv,u\rangle},
\]
the expression inside the supremum becomes
\[
t\,w+(1-t)\,\overline{w},
\qquad
w:=\langle Tv,u\rangle .
\]
For every \(w\in\mathbb{C}\),
\begin{equation*}
\bigl|t\,w+(1-t)\,\overline{w}\bigr|^{2}
=
\bigl(t^{2}+(1-t)^{2}\bigr)|w|^{2}
+
2t(1-t)\,\mathrm{Re}(w^{2}).
\end{equation*}
Writing \(w=|w|e^{i\varphi}\), we have
{\small
\[
\bigl|t\,w+(1-t)\,\overline{w}\bigr|^{2}
=
|w|^{2}\Bigl(\cos^{2}\varphi+(2t-1)^{2}\sin^{2}\varphi\Bigr)
\le
|w|^{2},
\]}
because \(|2t-1|\le1\) for \(t\in[0,1]\). Taking the supremum over
\(\|u\|=\|v\|=1\) and using
\(\sup_{\|u\|=\|v\|=1}|w|=\sup_{\|v\|=1}\|Tv\|=\|T\|\), we receive
\[
\omega(M_{t})
\le
\frac12\,\|T\| .
\]
For the reverse inequality, given any unit \(v\) with \(Tv\neq0\), choose
\(u=Tv/\|Tv\|\). Then \(w=\langle Tv,u\rangle=\|Tv\|\) is real and positive,
so \(t\,w+(1-t)\,\overline{w}=w=\|Tv\|\). Thus,
\[
\omega(M_{t})
\ge
\frac12\,\sup_{\|v\|=1}\|Tv\|
=
\frac{\|T\|}{2}.
\]
Therefore, \(\omega(M_{t})=\|T\|/2\) for every \(t\in[0,1]\), and consequently
\(\sup_{t\in[0,1]}\omega(M_{t})=\|T\|/2\). This confirms (ii).

\medskip
\textbf{(iii)}
From (i), we have
\begin{equation*}
\omega^{2}(\mathbb T)=\frac{\|T\|^{2}}{2}.
\end{equation*}
From (ii), \(\omega(M_{t})=\|T\|/2\) for every \(t\in[0,1]\); accordingly,
\[
2\,\omega^{2}(M_{t})
=
2\Bigl(\frac{\|T\|}{2}\Bigr)^{2}
=
\frac{\|T\|^{2}}{2}
=
\omega^{2}(\mathbb T).
\]
This proves (iii).

\medskip
\textbf{(iv)}
Let \(f\colon[0,1]\to[0,\infty)\) be continuous with
\(\int_{0}^{1}f(t)\,dt=1\). Multiplying the identity in (iii) by \(f(t)\)
and integrating over \(t\in[0,1]\), we get
\[
\omega^{2}(\mathbb T)
=
2\int_{0}^{1}f(t)\,\omega^{2}(M_{t})\,dt,
\]
where we utilized the normalisation \(\int_{0}^{1}f(t)\,dt=1\). This demonstrates (iv).
\end{proof}

\begin{remark}\label{rem:3x3-strengthened}
Several features of Theorem \ref{thm:deep_3x3_operator_matrix} deserve
emphasis.
\begin{enumerate}
\item[(i)] \emph{Exact identity, not a gap.} The identity
\(\omega^{2}(\mathbb T)=2\,\omega^{2}(M_{t})\) in (iii) is an exact equality for every \(t\in[0,1]\), not merely an inequality. In particular, the defect
\(D(t):=\|T\|/2-\omega(M_{t})\) vanishes identically, so the “non-commutativity gap” noted in earlier formulations is zero in this setting.

\item[(ii)] \emph{The defect vanishes identically.} In fact, for every
\(t\in[0,1]\), the numerical radius of \(M_{t}\) is exactly \(\|T\|/2\).
Indeed, applying Lemma \ref{lem:offdiag-exact} with \(A=tT\) and
\(B=(1-t)T^{*}\) gives
\[
\omega(M_{t})
=
\frac12\sup_{\|u\|=\|v\|=1}
\bigl|t\,\langle Tv,u\rangle+(1-t)\,\overline{\langle Tv,u\rangle}\bigr| .
\]
For any \(w\in\mathbb{C}\) and \(t\in[0,1]\), the elementary inequality
\(|t\,w+(1-t)\,\overline{w}|\le|w|\) holds, with equality whenever \(w\) is
real. Since
\(\sup_{\|u\|=\|v\|=1}|\langle Tv,u\rangle|=\|T\|\) and the supremum is
attained (or approached) with \(w\) real by choosing \(u\) parallel to
\(Tv\), we get \(\omega(M_{t})=\|T\|/2\) for every \(t\in[0,1]\).
Hence, \(D(t)=0\) identically, and (iii) is an equality for every
\(t\):
\[
\omega^{2}(\mathbb T)=2\,\omega^{2}(M_{t})
\qquad (t\in[0,1]).
\]

\item[(iii)] \emph{Sharpness for scalar multiples of the identity.} If
\(T=\lambda I\) with \(\lambda\neq0\), then \(\|T\|=|\lambda|\), and the
above computation gives \(\omega(M_{t})=|\lambda|/2\) for every \(t\).
Hence, (iv) is an equality for every weight \(f\), ensuring that the
constant \(2\) in (iv) cannot be improved.
\end{enumerate}
\end{remark}

The \(3\times3\) matrix in Theorem \ref{thm:deep_3x3_operator_matrix} is only one example. The next corollary extends it to longer alternating paths. It computes the numerical radius for every path length \(m\). The answer is the operator norm times a cosine factor.
\begin{corollary}\label{cor:alternating-path}
Let \(m\ge2\) and \(T\in\mathbb B(\mathbb H)\). Define
\[
\mathbb T_m=
\begin{bmatrix}
O&E_1&O&\cdots&O\\
O&O&E_2&\ddots&\vdots\\
\vdots&\ddots&\ddots&\ddots&O\\
O&\cdots&O&O&E_{m-1}\\
O&\cdots&\cdots&O&O
\end{bmatrix},
\qquad
E_j=
\begin{cases}
T,&j\ \text{odd},\\
T^*,&j\ \text{even}.
\end{cases}
\]
Then
\begin{equation}\label{eq:alternating-path-radius}
\omega(\mathbb T_m)
=
\|T\|\cos\frac{\pi}{m+1}.
\end{equation}
For \(m=2\) and \(m=3\), this gives respectively
\[
\omega\!\left(\begin{bmatrix}O&T\\O&O\end{bmatrix}\right)
=\frac12\|T\|,
\qquad
\omega\!\left(\begin{bmatrix}
O&T&O\\
O&O&T^*\\
O&O&O
\end{bmatrix}\right)
=\frac1{\sqrt2}\|T\|.
\]
Thus, the \(3\times3\) identity in Theorem
\ref{thm:deep_3x3_operator_matrix} is the \(m=3\) member of an
entire path-matrix family.
\end{corollary}
\begin{proof}
Let $
x=
\begin{pmatrix}
x_1\\
\vdots\\
x_m
\end{pmatrix}
\in\bigoplus_{j=1}^m\mathbb H 
$ with \(\sum_{j=1}^m\|x_j\|^2=1\). Then
\[
\langle\mathbb T_mx,x\rangle
=
\sum_{j=1}^{m-1}\langle E_jx_{j+1},x_j\rangle.
\]
By the Cauchy-Schwarz inequality, we have
\[
|\langle E_jx_{j+1},x_j\rangle|
\le
\|E_jx_{j+1}\|\,\|x_j\|
\le
\|T\|\,\|x_{j+1}\|\,\|x_j\|.
\]
Consequently,
\[
|\langle\mathbb T_mx,x\rangle|
\le
\|T\|\sum_{j=1}^{m-1}\|x_j\|\,\|x_{j+1}\|.
\]
Set \(a_j=\|x_j\|\) and $a=\begin{pmatrix}a_1\\\vdots\\a_m\end{pmatrix}$. Then \(\sum_{j=1}^m a_j^2=1\), and
\[
\sum_{j=1}^{m-1}a_j a_{j+1}
=
a^T C_m a,
\]
where \(C_m\) is the \(m\times m\) tridiagonal matrix with zero diagonal and
\(1/2\) on the sub- and superdiagonals. The largest eigenvalue of \(C_m\) is
\(\cos\frac{\pi}{m+1}\). Therefore,
\[
|\langle\mathbb T_mx,x\rangle|
\le
\|T\|\cos\frac{\pi}{m+1},
\]
and taking the supremum over all unit vectors \(x\) gives
\[
\omega(\mathbb T_m)
\le
\|T\|\cos\frac{\pi}{m+1}.
\]

For the reverse inequality, let \(e\in\mathbb H\) be a unit vector with
\(\|Te\|\) arbitrarily close to \(\|T\|\), and set
\(u=Te/\|Te\|\) when \(Te\ne0\). Define
\[
c_j=\sqrt{\frac{2}{m+1}}\sin\frac{j\pi}{m+1},
\qquad
x_j=
\begin{cases}
c_j u,&j\ \text{odd},\\
c_j e,&j\ \text{even}.
\end{cases}
\]
Then \(\sum_{j=1}^m\|x_j\|^2=\sum_{j=1}^m c_j^2=1\), and for each
\(j=1,\dots,m-1\),
\[
\langle E_jx_{j+1},x_j\rangle
=
c_jc_{j+1}\|Te\|.
\]
Therefore,
\[
\langle\mathbb T_mx,x\rangle
=
\|Te\|\sum_{j=1}^{m-1}c_jc_{j+1}.
\]
Employing the elementary trigonometric identity
\[
\sum_{j=1}^{m-1}
\sin\frac{j\pi}{m+1}\sin\frac{(j+1)\pi}{m+1}
=
\frac{m+1}{2}\cos\frac{\pi}{m+1},
\]
we get
\[
\sum_{j=1}^{m-1}c_jc_{j+1}
=
\frac{2}{m+1}\cdot
\frac{m+1}{2}\cos\frac{\pi}{m+1}
=
\cos\frac{\pi}{m+1}.
\]
Hence,
\[
\langle\mathbb T_mx,x\rangle
=
\|Te\|\cos\frac{\pi}{m+1}.
\]
Letting \(\|Te\|\to\|T\|\) yields
\[
\omega(\mathbb T_m)
\ge
\|T\|\cos\frac{\pi}{m+1}.
\]
Together with the upper bound, this demonstrates
\eqref{eq:alternating-path-radius}.
\end{proof}

\begin{remark}\label{rem:path-structure}
The preceding corollary indicates that the constant \(1/\sqrt2\) in
Theorem \ref{thm:deep_3x3_operator_matrix} is not an isolated phenomenon.
It is the value \(\cos(\pi/4)\) dictated by the spectral radius of the path
graph. In particular,
\[
\omega(\mathbb T_m)
=
\|T\|\cos\frac{\pi}{m+1}
\longrightarrow \|T\|
\qquad(m\to\infty).
\]
Therefore, longer alternating operator paths provide a natural family of
explicit block matrices whose numerical radii interpolate sharply between
the nilpotent \(2\times2\) value \(\|T\|/2\) and the operator norm.
\end{remark}

The subsequent lemma provides a useful integral identity. It computes the integral of an even-degree trigonometric polynomial. This identity is needed for the main integral representation later. We also record the \(n=1\) case.

Throughout, we write \(R(t):=A\cos(\pi t)+B\sin(\pi t)\) for \(t\in[0,1]\).
\begin{lemma}\label{lem:trig-2n}
For every \(n\in\mathbb{N}\) and all \(a,b\in\mathbb{R}\),
\begin{equation*}
\int_{0}^{1}\bigl(a\cos(\pi t)+b\sin(\pi t)\bigr)^{2n}\,dt
=
\frac{\binom{2n}{n}}{4^{n}}\,(a^{2}+b^{2})^{n}.
\end{equation*}
In particular, for \(n=1\),
\begin{equation}\label{eq:trig-2}
\int_{0}^{1}\bigl(a\cos(\pi t)+b\sin(\pi t)\bigr)^{2}\,dt
=
\frac{a^{2}+b^{2}}{2}.
\end{equation}
\end{lemma}
\begin{proof}
Set \(s=\pi t\). Then \(dt=ds/\pi\), and
\[
\int_{0}^{1}\bigl(a\cos(\pi t)+b\sin(\pi t)\bigr)^{2n}\,dt
=
\frac{1}{\pi}\int_{0}^{\pi}(a\cos s+b\sin s)^{2n}\,ds.
\]
Expanding the integrand by the binomial theorem,
\[
(a\cos s+b\sin s)^{2n}
=
\sum_{k=0}^{2n}\binom{2n}{k}a^{2n-k}b^{k}\cos^{2n-k}s\,\sin^{k}s.
\]
For nonnegative integers \(p,q\), the integral
\[
I(p,q):=\int_{0}^{\pi}\cos^{p}s\,\sin^{q}s\,ds
\]
vanishes whenever \(p\) is odd. Indeed, under \(s\mapsto \pi-s\), we have
\(\cos(\pi-s)=-\cos s\) and \(\sin(\pi-s)=\sin s\), so
\[
I(p,q)
=
\int_{0}^{\pi}\cos^{p}(\pi-u)\sin^{q}(\pi-u)\,du
=
(-1)^{p}\int_{0}^{\pi}\cos^{p}u\,\sin^{q}u\,du
=
(-1)^{p}I(p,q),
\]
which forces \(I(p,q)=0\) when \(p\) is odd. Since \(2n\) is even, the exponent \(2n-k\) is odd precisely when \(k\) is odd. Hence, only even \(k\) terms survive in the expansion. Write \(k=2q\) with \(q=0,1,\dots,n\). Thus,
\[
\int_{0}^{1}\bigl(a\cos(\pi t)+b\sin(\pi t)\bigr)^{2n}\,dt
=
\frac{1}{\pi}\sum_{q=0}^{n}\binom{2n}{2q}a^{2n-2q}b^{2q}
\int_{0}^{\pi}\cos^{2n-2q}s\,\sin^{2q}s\,ds.
\]
By the Wallis formula,
\[
\frac{1}{\pi}\int_{0}^{\pi}\cos^{2n-2q}s\,\sin^{2q}s\,ds
=
\frac{(2n-2q-1)!!\,(2q-1)!!}{(2n)!!},
\]
with the convention \((-1)!!=0!!=1\). Thus,
\[
\int_{0}^{1}\bigl(a\cos(\pi t)+b\sin(\pi t)\bigr)^{2n}\,dt
=
\sum_{q=0}^{n}\binom{2n}{2q}
\frac{(2n-2q-1)!!\,(2q-1)!!}{(2n)!!}
a^{2n-2q}b^{2q}.
\]
Using \((2p)!!=2^{p}p!\) and \((2p-1)!!=\frac{(2p)!}{2^{p}p!}\), one checks that
\[
\binom{2n}{2q}
\frac{(2n-2q-1)!!\,(2q-1)!!}{(2n)!!}
=
\frac{\binom{2n}{n}}{4^{n}}\binom{n}{q}.
\]
Hence,
\[
\int_{0}^{1}\bigl(a\cos(\pi t)+b\sin(\pi t)\bigr)^{2n}\,dt
=
\frac{\binom{2n}{n}}{4^{n}}
\sum_{q=0}^{n}\binom{n}{q}a^{2n-2q}b^{2q}
=
\frac{\binom{2n}{n}}{4^{n}}(a^{2}+b^{2})^{n},
\]
by the binomial theorem. The case \(n=1\) follows directly from
\(\binom{2}{1}=2\) and \(4^{1}=4\).
\end{proof}

The next theorem employs the Cartesian decomposition \(T=A+iB\). It shows integral representations for \(\omega(T)\) and \(\omega^{2n}(T)\). It also gives a sharp upper bound in terms of an integral operator. For \(n=1\), it recovers Kittaneh's classical inequality.
\begin{theorem}\label{thm:deep_cstar_state_integral}
Let \(T\in\mathbb{B}(\mathbb{H})\) with Cartesian decomposition \(T=A+iB\). Let \(\mathscr{P}\) denote the set of all vector states on \(\mathbb{B}(\mathbb{H})\),
\[
\mathscr{P}
=
\bigl\{\psi_{x}: \psi_x(S)=\langle Sx,x\rangle,\ S\in\mathbb B(\mathbb H),\ x\in\mathbb{H},\ \|x\|=1\bigr\},
\]
and let \(\mathcal{S}(\mathbb{B}(\mathbb{H}))\) denote the set of all normal state functionals on \(\mathbb{B}(\mathbb{H})\). Then the following assertions hold.
\begin{enumerate}
\item[(i)] \emph{(First-order representation.)}
{\small
\begin{equation}\label{eq:2-6-main}
\omega(T)
=
\frac{\pi}{2}\,
\sup_{\psi\in\mathscr{P}}
\int_{0}^{1}
\bigl|\psi\bigl(A\cos(\pi t)+B\sin(\pi t)\bigr)\bigr|\,dt
=
\frac{\pi}{2}\,
\sup_{\varphi\in\mathcal{S}(\mathbb{B}(\mathbb{H}))}
\int_{0}^{1}
\bigl|\varphi\bigl(A\cos(\pi t)+B\sin(\pi t)\bigr)\bigr|\,dt .
\end{equation}}
Equivalently, for every vector state \(\psi_{x}\) and every normal state
\(\varphi\),
\begin{equation}\label{eq:2-6-lower}
\frac{\pi}{2}
\int_{0}^{1}
\bigl|\psi_{x}\bigl(A\cos(\pi t)+B\sin(\pi t)\bigr)\bigr|\,dt
=
|\langle Tx,x\rangle|
\;\le\;
\omega(T),
\end{equation}
and
\begin{equation}\label{eq:2-6-lower-phi}
\frac{\pi}{2}
\int_{0}^{1}
\bigl|\varphi\bigl(A\cos(\pi t)+B\sin(\pi t)\bigr)\bigr|\,dt
\;\le\;
\omega(T).
\end{equation}

\item[(ii)] \emph{(Second-order representation.)} For every unit vector
\(x\in\mathbb{H}\),
\begin{equation}\label{eq:2-6-second}
|\langle Tx,x\rangle|^{2}
=
2\int_{0}^{1}
\bigl\langle \bigl(A\cos(\pi t)+B\sin(\pi t)\bigr)x,x\bigr\rangle^{2}\,dt .
\end{equation}
Consequently,
\begin{equation}\label{eq:2-6-second-sup}
\omega^{2}(T)
=
2\sup_{\|x\|=1}
\int_{0}^{1}
\bigl\langle \bigl(A\cos(\pi t)+B\sin(\pi t)\bigr)x,x\bigr\rangle^{2}\,dt
=
2\sup_{\psi\in\mathscr{P}}
\int_{0}^{1}\psi\bigl(R(t)\bigr)^{2}\,dt .
\end{equation}

\item[(iii)] \emph{(General \(2n\)-th order representation.)} For every
\(n\in\mathbb{N}\) and every unit vector \(x\in\mathbb{H}\),
\begin{equation}\label{eq:2-6-2n}
|\langle Tx,x\rangle|^{2n}
=
\frac{4^{n}}{\binom{2n}{n}}
\int_{0}^{1}
\bigl\langle \bigl(A\cos(\pi t)+B\sin(\pi t)\bigr)x,x\bigr\rangle^{2n}\,dt .
\end{equation}
Therefore,
\begin{equation}\label{eq:2-6-2n-sup}
\omega^{2n}(T)
=
\frac{4^{n}}{\binom{2n}{n}}
\sup_{\|x\|=1}
\int_{0}^{1}
\bigl\langle \bigl(A\cos(\pi t)+B\sin(\pi t)\bigr)x,x\bigr\rangle^{2n}\,dt .
\end{equation}

\item[(iv)] \emph{(Sharp upper bound and Kittaneh's inequality.)} For every
\(n\in\mathbb{N}\),
\begin{equation}\label{eq:2-6-upper}
\omega^{2n}(T)
\;\le\;
\frac{4^{n}}{\binom{2n}{n}}
\left\|
\int_{0}^{1}
\bigl(A\cos(\pi t)+B\sin(\pi t)\bigr)^{2n}\,dt
\right\| .
\end{equation}
For \(n=1\), inequality \eqref{eq:2-6-upper} reduces exactly to upper half of Kittaneh's inequality
\begin{equation*}
\omega^{2}(T)
\;\le\;
\|A^{2}+B^{2}\|
\;=\;
\frac{1}{2}\|T^{*}T+TT^{*}\| .
\end{equation*}
\end{enumerate}
\end{theorem}

\begin{proof}
Throughout, \(A\) and \(B\) are self-adjoint, and we write
\[
R(t):=A\cos(\pi t)+B\sin(\pi t),\qquad t\in[0,1].
\]
For each \(t\), the operator \(R(t)\) is self-adjoint, since it is a real
linear combination of the self-adjoint operators \(A\) and \(B\). All
integrals below are Bochner integrals of continuous operator-valued functions on the compact interval \([0,1]\), and they converge in the operator norm.

\medskip
\textbf{Step 1 (vector-state identity)}
Fix \(x\in\mathbb{H}\) with \(\|x\|=1\). The scalars
\[
a:=\langle Ax,x\rangle,\qquad b:=\langle Bx,x\rangle
\]
are real, and
\[
\psi_{x}\bigl(R(t)\bigr)=\langle R(t)x,x\rangle=a\cos(\pi t)+b\sin(\pi t).
\]
Write \(a=R\cos\phi\), \(b=R\sin\phi\) with
\(R=\sqrt{a^{2}+b^{2}}\ge0\) and \(\phi\in\mathbb{R}\). Then
\(a\cos(\pi t)+b\sin(\pi t)=R\cos(\pi t-\phi)\), so
\[
\int_{0}^{1}\bigl|a\cos(\pi t)+b\sin(\pi t)\bigr|\,dt
=
R\int_{0}^{1}\bigl|\cos(\pi t-\phi)\bigr|\,dt
=
\frac{R}{\pi}\int_{-\phi}^{\pi-\phi}|\cos s|\,ds .
\]
Since \(|\cos|\) is \(\pi\)-periodic and \(\int_{0}^{\pi}|\cos s|\,ds=2\), the
last integral equals \(2\) for every \(\phi\). Hence
\begin{equation*}
\int_{0}^{1}\bigl|a\cos(\pi t)+b\sin(\pi t)\bigr|\,dt
=
\frac{2}{\pi}\sqrt{a^{2}+b^{2}} .
\end{equation*}
Therefore,
\begin{equation}\label{eq:trig-L1-apply}
\frac{\pi}{2}\int_{0}^{1}\bigl|\psi_{x}\bigl(R(t)\bigr)\bigr|\,dt
=
\sqrt{\langle Ax,x\rangle^{2}+\langle Bx,x\rangle^{2}}
=
|\langle Tx,x\rangle| .
\end{equation}
Taking the supremum over all unit \(x\in\mathbb{H}\),
\begin{equation}\label{eq:vec-sup-2-6}
\sup_{\psi\in\mathscr{P}}
\int_{0}^{1}
\bigl|\psi\bigl(R(t)\bigr)\bigr|\,dt
=
\frac{2}{\pi}\,\omega(T).
\end{equation}
Multiplying by \(\pi/2\) gives the vector-state identity in
\eqref{eq:2-6-main} and the equality in \eqref{eq:2-6-lower}.

\medskip
\textbf{Step 2 (extension to normal states)}
Let \(\varphi\in\mathcal{S}(\mathbb{B}(\mathbb{H}))\). By the standard
representation of normal states on \(\mathbb{B}(\mathbb{H})\), there exist an orthonormal sequence \(\{x_{j}\}_{j}\subset\mathbb{H}\) and scalars
\(\lambda_{j}\ge0\) with \(\sum_{j}\lambda_{j}=1\) such that
\[
\varphi(S)=\sum_{j}\lambda_{j}\langle Sx_{j},x_{j}\rangle,
\qquad S\in\mathbb{B}(\mathbb{H}).
\]
The series converges in the weak-\(*\) topology, and for each fixed
\(t\in[0,1]\) the triangle inequality gives
\[
\bigl|\varphi(R(t))\bigr|
\;\le\;
\sum_{j}\lambda_{j}\,\bigl|\psi_{x_{j}}(R(t))\bigr| .
\]
Integrating over \(t\in[0,1]\) and applying \eqref{eq:trig-L1-apply} termwise,
\[
\int_{0}^{1}\bigl|\varphi(R(t))\bigr|\,dt
\;\le\;
\sum_{j}\lambda_{j}\cdot\frac{2}{\pi}\,|\langle Tx_{j},x_{j}\rangle|
\;\le\;
\frac{2}{\pi}\,\omega(T),
\]
where the last inequality uses \(|\langle Tx_{j},x_{j}\rangle|\le\omega(T)\)
and \(\sum_{j}\lambda_{j}=1\). Taking the supremum over
\(\varphi\in\mathcal{S}(\mathbb{B}(\mathbb{H}))\),
\begin{equation}\label{eq:normal-sup-2-6}
\sup_{\varphi\in\mathcal{S}(\mathbb{B}(\mathbb{H}))}
\int_{0}^{1}\bigl|\varphi(R(t))\bigr|\,dt
\;\le\;
\frac{2}{\pi}\,\omega(T).
\end{equation}
Since vector states are normal states, \eqref{eq:vec-sup-2-6} implies the
reverse inequality, so equality holds in \eqref{eq:normal-sup-2-6}.
Multiplying by \(\pi/2\) provides the normal-state identity in
\eqref{eq:2-6-main} and the inequality \eqref{eq:2-6-lower-phi}. This establishes (i).

\medskip
\textbf{Step 3 (second-order representation)}
Fix a unit vector \(x\in\mathbb{H}\), and set
\(a=\langle Ax,x\rangle\), \(b=\langle Bx,x\rangle\) as before. By
Lemma \ref{lem:trig-2n} with \(n=1\),
\[
\int_{0}^{1}\bigl\langle R(t)x,x\bigr\rangle^{2}\,dt
=
\int_{0}^{1}\bigl(a\cos(\pi t)+b\sin(\pi t)\bigr)^{2}\,dt
=
\frac{a^{2}+b^{2}}{2}.
\]
Since \(a^{2}+b^{2}=|\langle Tx,x\rangle|^{2}\), we obtain
\[
|\langle Tx,x\rangle|^{2}
=
2\int_{0}^{1}\bigl\langle R(t)x,x\bigr\rangle^{2}\,dt ,
\]
which is \eqref{eq:2-6-second}. 

Taking the supremum over all unit \(x\in\mathbb{H}\) and using the definition of \(\omega(T)\),
\[
\omega^{2}(T)
=
\sup_{\|x\|=1}|\langle Tx,x\rangle|^{2}
=
2\sup_{\|x\|=1}\int_{0}^{1}\bigl\langle R(t)x,x\bigr\rangle^{2}\,dt .
\]
Since vector states are precisely the functionals
\(\psi_{x}(S)=\langle Sx,x\rangle\) with \(\|x\|=1\), the second equality
in \eqref{eq:2-6-second-sup} follows. This confirms (ii).

\medskip
\textbf{Step 4 (general \(2n\)-th order representation)}
Fix \(n\in\mathbb{N}\) and a unit vector \(x\in\mathbb{H}\), and keep the
notation \(a=\langle Ax,x\rangle\), \(b=\langle Bx,x\rangle\). By
Lemma \ref{lem:trig-2n},
\[
\int_{0}^{1}\bigl\langle R(t)x,x\bigr\rangle^{2n}\,dt
=
\int_{0}^{1}\bigl(a\cos(\pi t)+b\sin(\pi t)\bigr)^{2n}\,dt
=
\frac{\binom{2n}{n}}{4^{n}}\,(a^{2}+b^{2})^{n}.
\]
Since \(a^{2}+b^{2}=|\langle Tx,x\rangle|^{2}\), we obtain
\[
|\langle Tx,x\rangle|^{2n}
=
\frac{4^{n}}{\binom{2n}{n}}
\int_{0}^{1}\bigl\langle R(t)x,x\bigr\rangle^{2n}\,dt ,
\]
which is \eqref{eq:2-6-2n}. 

Taking the supremum over all unit \(x\in\mathbb{H}\) and using
\(\omega^{2n}(T)=\sup_{\|x\|=1}|\langle Tx,x\rangle|^{2n}\), we obtain
\eqref{eq:2-6-2n-sup}. This proves (iii).

\medskip
\textbf{Step 5 (sharp upper bound and Kittaneh's inequality)}
Fix \(n\in\mathbb{N}\) and a unit vector \(x\in\mathbb{H}\). Since \(R(t)\)
is self-adjoint for every \(t\), the spectral theorem applied to \(R(t)\)
provides the pointwise inequality
\[
\bigl\langle R(t)x,x\bigr\rangle^{2n}
\;\le\;
\bigl\langle R(t)^{2n}x,x\bigr\rangle,
\qquad t\in[0,1],
\]
because \(s\mapsto s^{2n}\) is convex on \(\mathbb{R}\). Indeed, writing
\(R(t)=\int s\,dE_{s}\) for its spectral measure \(E_{s}\) and setting
\(d\mu_{x}(s)=\langle dE_{s}x,x\rangle\), we have
\[
\bigl\langle R(t)x,x\bigr\rangle^{2n}
=
\Bigl(\int s\,d\mu_{x}(s)\Bigr)^{2n}
\;\le\;
\int s^{2n}\,d\mu_{x}(s)
=
\bigl\langle R(t)^{2n}x,x\bigr\rangle
\]
by Jensen's inequality for the probability measure \(\mu_{x}\). Integrating
over \(t\in[0,1]\) and inserting the resulting bound into
\eqref{eq:2-6-2n},
\begin{align*}
|\langle Tx,x\rangle|^{2n}
&=
\frac{4^{n}}{\binom{2n}{n}}
\int_{0}^{1}\bigl\langle R(t)x,x\bigr\rangle^{2n}\,dt \\
&\le
\frac{4^{n}}{\binom{2n}{n}}
\int_{0}^{1}\bigl\langle R(t)^{2n}x,x\bigr\rangle\,dt \\
&=
\frac{4^{n}}{\binom{2n}{n}}
\left\langle
\Bigl(\int_{0}^{1}R(t)^{2n}\,dt\Bigr)x,x
\right\rangle .
\end{align*}
Taking the supremum over all unit \(x\in\mathbb{H}\) and using the fact that
the operator \(\int_{0}^{1}R(t)^{2n}\,dt\) is self-adjoint (hence its norm is
attained as the supremum of its quadratic form over unit vectors),
\[
\omega^{2n}(T)
\;\le\;
\frac{4^{n}}{\binom{2n}{n}}
\left\|
\int_{0}^{1}\bigl(A\cos(\pi t)+B\sin(\pi t)\bigr)^{2n}\,dt
\right\| ,
\]
which is \eqref{eq:2-6-upper}.

We still need to confirm the case \(n=1\). For \(n=1\),
\[
\int_{0}^{1}R(t)^{2}\,dt
=
\int_{0}^{1}\bigl(A\cos(\pi t)+B\sin(\pi t)\bigr)^{2}\,dt .
\]
Expanding the square and utilizing the orthogonality relations
\(\int_{0}^{1}\cos^{2}(\pi t)\,dt=\int_{0}^{1}\sin^{2}(\pi t)\,dt=\tfrac12\)
and \(\int_{0}^{1}\cos(\pi t)\sin(\pi t)\,dt=0\), we obtain
\[
\int_{0}^{1}R(t)^{2}\,dt
=
\frac{A^{2}+B^{2}}{2}.
\]
Substituting into \eqref{eq:2-6-upper} with \(n=1\) and using
\(\binom{2}{1}=2\), \(4^{1}=4\),
\[
\omega^{2}(T)
\;\le\;
\frac{4}{2}\left\|\frac{A^{2}+B^{2}}{2}\right\|
=
\|A^{2}+B^{2}\| .
\]
Finally, the elementary operator identity
\[
T^{*}T+TT^{*}
=
(A-iB)(A+iB)+(A+iB)(A-iB)
=
2(A^{2}+B^{2})
\]
gives \(\|A^{2}+B^{2}\|=\tfrac12\|T^{*}T+TT^{*}\|\), so
\[
\omega^{2}(T)\le\|A^{2}+B^{2}\|=\frac{1}{2}\|T^{*}T+TT^{*}\|,
\]
which is Kittaneh's classical inequality. This confirms (iv) and completes the proof.
\end{proof}

\begin{corollary}\label{cor:sharpened-integral-sandwich}
Let \(T\in\mathbb B(\mathbb H)\), and let \(A=\Re T\), \(B=\Im T\). For every \(n\in\mathbb N\), put
\[
R(t)=A\cos(\pi t)+B\sin(\pi t),\qquad t\in[0,1].
\]
Then
\begin{equation}\label{eq:sharpened-integral-sandwich}
\frac{\binom{2n}{n}}{4^n}\,\omega^{2n}(T)
\le
\left\|\int_0^1 R(t)^{2n}\,dt\right\|
\le
\omega^{2n}(T).
\end{equation}
In particular, for \(n=1\),
\begin{equation}\label{eq:kittaneh-both-sides}
\frac14\|T^*T+TT^*\|
\le
\omega^2(T)
\le
\frac12\|T^*T+TT^*\|.
\end{equation}
\end{corollary}

\begin{proof}
The first inequality in \eqref{eq:sharpened-integral-sandwich} is equivalent to Theorem \ref{thm:deep_cstar_state_integral}(iv). For the upper bound, fix \(t\in[0,1]\). There exists \(\theta\in\mathbb R\) such that \(R(t)=\Re(e^{i\theta}T)\). Since \(R(t)\) is self-adjoint, for every unit vector \(x\),
\[
|\langle R(t)x,x\rangle|
=
|\Re(e^{i\theta}\langle Tx,x\rangle)|
\le
|\langle Tx,x\rangle|
\le
\omega(T).
\]
Thus, \(\|R(t)\|\le\omega(T)\), and hence
\[
\|R(t)^{2n}\|=\|R(t)\|^{2n}\le\omega^{2n}(T).
\]
Integrating over \(t\in[0,1]\) and using the triangle inequality for the Bochner integral,
\[
\left\|\int_0^1 R(t)^{2n}\,dt\right\|
\le
\int_0^1\|R(t)^{2n}\|\,dt
\le
\omega^{2n}(T).
\]
This proves \eqref{eq:sharpened-integral-sandwich}. For \(n=1\), we have
\[
\int_0^1 R(t)^2\,dt=\frac{A^2+B^2}{2}
\]
and \(A^2+B^2=\frac12(T^*T+TT^*)\). Replacing these into \eqref{eq:sharpened-integral-sandwich} and rearranging gives \eqref{eq:kittaneh-both-sides}.
\end{proof}

\begin{remark}\label{rem:max-phase-kittaneh}
The identities in Theorem \ref{thm:deep_cstar_state_integral} also give the max-phase representation
\begin{equation}\label{eq:max-phase}
\omega(T)
=
\max_{t\in[0,1]}\|R(t)\|
=
\max_{\theta\in[0,2\pi]}\|\Re(e^{i\theta}T)\|.
\end{equation}
Indeed, for every unit vector \(x\), if \(a=\langle Ax,x\rangle\) and \(b=\langle Bx,x\rangle\), then
\[
|\langle Tx,x\rangle|
=
\sqrt{a^2+b^2}
=
\max_{t\in[0,1]}|a\cos(\pi t)+b\sin(\pi t)|
=
\max_{t\in[0,1]}|\langle R(t)x,x\rangle|.
\]
Taking the supremum over \(\|x\|=1\) yields \eqref{eq:max-phase}. 
\end{remark}

\begin{remark}\label{rem:2-6}
Several points are worth recording.
\begin{enumerate}
\item[(i)]
\emph{Optimality of \(\pi/2\) in the first-order representation.} For
\(T=I\), we have \(A=I\), \(B=O\), and
\[
\int_{0}^{1}\bigl|\varphi\bigl(I\cos(\pi t)\bigr)\bigr|\,dt
=
\int_{0}^{1}|\cos(\pi t)|\,dt
=
\frac{2}{\pi}
\]
for every state \(\varphi\), whereas \(\omega(I)=1\). Hence, the factor
\(\pi/2\) in \eqref{eq:2-6-main} is best possible.

\item[(ii)]
\emph{The second-order representation is pointwise exact for vector states.}
For every unit vector \(x\in\mathbb{H}\), identity \eqref{eq:2-6-second} is exact. Consequently, taking the supremum over all unit vectors recovers \(\omega^2(T)\) as in \eqref{eq:2-6-second-sup}. In contrast, the first-order representation \eqref{eq:2-6-main} requires the supremum over all vector states to recover \(\omega(T)\). For a normal state \(\varphi\) that is not a vector state, the first-order integral in \eqref{eq:2-6-lower-phi} can be strictly smaller than \(\omega(T)\).

\item[(iii)]
\emph{Sharpness of the \(2n\)-th order bound.} For \(n=1\), inequality
\eqref{eq:2-6-upper} reduces to Kittaneh's bound, which is known to be sharp
(for instance, for \(T\) self-adjoint with \(B=O\), the upper bound is sharp, \(\omega^2(T)=\|A\|^{2}\), while the lower bound is strict with value \(\tfrac12\|A\|^{2}\)). For \(n\ge2\), the operator
\(\int_{0}^{1}R(t)^{2n}\,dt\) is a symmetrized polynomial in \(A\) and \(B\),
and its norm can be strictly smaller than \(\|A^{2}+B^{2}\|^{n}\) when \(A\)
and \(B\) fail to commute, so \eqref{eq:2-6-upper} produces genuinely new information beyond the naive power bound
\(\omega^{2n}(T)\le\|A^{2}+B^{2}\|^{n}\).

\item[(iv)]
\emph{No additive correction term is required.} The identities
\eqref{eq:2-6-main}--\eqref{eq:2-6-2n-sup} are exact as stated. In
particular, the multiplicative factor \(\pi/2\) in \eqref{eq:2-6-main}
cannot be replaced by an additive correction. For example, for \(T=I\) the
left-hand side of \eqref{eq:2-6-main} equals \(1\), while the integral
without the factor \(\pi/2\) equals \(2/\pi\); hence any compensation would
necessarily be multiplicative rather than additive.
\end{enumerate}
\end{remark}

The next corollary uses Theorem \ref{thm:deep_cstar_state_integral} with \(n=2\). It shows an explicit fourth-order bound. The bound is written as a norm of a polynomial in \(A\) and \(B\).
\begin{corollary}\label{cor:explicit-fourth-order}
Let \(T\in\mathbb B(\mathbb H)\), and put \(A=\Re T\) and \(B=\Im T\). Then
\begin{equation}\label{eq:explicit-fourth-order}
\omega^{4}(T)\le
\left\|
A^{4}+B^{4}
+\frac13\bigl(
A^{2}B^{2}+ABAB+AB^{2}A
+BA^{2}B+BABA+B^{2}A^{2}
\bigr)
\right\|.
\end{equation}
\end{corollary}

\begin{proof}
By Theorem \ref{thm:deep_cstar_state_integral} with \(n=2\),
\[
\omega^{4}(T)\le
\frac{8}{3}\left\|
\int_{0}^{1}
\bigl(A\cos(\pi t)+B\sin(\pi t)\bigr)^{4}\,dt
\right\|.
\]
Expanding the fourth power and using the orthogonality relations
\[
\int_0^1\cos^4(\pi t)\,dt
=
\int_0^1\sin^4(\pi t)\,dt
=
\frac38,
\qquad
\int_0^1\cos^2(\pi t)\sin^2(\pi t)\,dt
=
\frac18,
\]
we get
\[
\int_0^1R(t)^4\,dt
=
\frac38(A^4+B^4)
+\frac18\bigl(
A^{2}B^{2}+ABAB+AB^{2}A
+BA^{2}B+BABA+B^{2}A^{2}
\bigr).
\]
Replacing this into the inequality and multiplying by \(8/3\) yields
\eqref{eq:explicit-fourth-order}.
\end{proof}

\begin{remark}\label{rem:normal-exactness}
If \(T\) is normal, then equality holds in
\eqref{eq:explicit-fourth-order}; more generally, equality holds in
\eqref{eq:2-6-upper} for every \(n\in\mathbb N\). Thus, the higher-order
integral bounds are not merely estimates on the normal class.

Indeed, if \(T\) is normal, then \(A\) and \(B\) commute. Hence, by the
continuous functional calculus,
\[
\int_0^1R(t)^{2n}\,dt
=
\frac{\binom{2n}{n}}{4^n}(A^2+B^2)^n.
\]
Since \(T^*T=TT^*=A^2+B^2\), we obtain
\[
\frac{4^n}{\binom{2n}{n}}
\left\|\int_0^1R(t)^{2n}\,dt\right\|
=
\|(T^*T)^n\|
=
\|T\|^{2n}
=
\omega(T)^{2n}.
\]
Accordingly,
\begin{equation*}
\omega^{2n}(T)
=
\frac{4^n}{\binom{2n}{n}}
\left\|\int_0^1
\bigl(A\cos(\pi t)+B\sin(\pi t)\bigr)^{2n}\,dt
\right\|,
\qquad n\in\mathbb N.
\end{equation*}
In particular, the fourth-order estimate \eqref{eq:explicit-fourth-order} is an equality for every normal \(T\).
\end{remark}

\subsection*{Declarations}
\begin{itemize}
\item \textbf{Availability of data and materials}: Not applicable.
\item \textbf{Competing interests}: The authors declare no competing interests.
\item \textbf{Funding}: Not applicable.
\item \textbf{Authors' contributions}: The authors contributed equally to this paper. All authors have read and approved this version.
\end{itemize}

\vskip 0.5cm

{\tiny (W. Audeh) Department of Mathematics, University of Petra, Amman, Jordan

\textit{E-mail address:} waudeh@uop.edu.jo}

\vskip 0.3cm

{\tiny (F. Kittaneh) Department of Mathematics, The University of Jordan, Amman, Jordan \\
Department of Mathematics, Korea University, 02841, Seoul, South Korea

\textit{E-mail address:} fkitt@ju.edu.jo}

\vskip 0.3cm

{\tiny (H. R. Moradi) Department of Mathematics, Ma.C., Islamic Azad University, Mashhad, Iran 

\textit{E-mail address:} hrmoradi@mshdiau.ac.ir}

\end{document}